\documentclass[12pt]{amsart}
\usepackage{amssymb,amsmath,amsthm,mathrsfs,multirow,xcolor,framed,url}
\usepackage[pdftex,
         pdfauthor={Rong Chen},
         pdftitle={temp},
         pdfsubject={MATHEMATICS},
         pdfkeywords={temp},
         pdfproducer={Latex with hyperref},
         pdfcreator={pdflatex}]{hyperref}
\usepackage [latin1]{inputenc}
\newtheorem{theorem}{Theorem}[section]
\newtheorem{lemma}[theorem]{Lemma}

\newtheorem{conj}[theorem]{Conjecture}

\theoremstyle{definition}

\theoremstyle{remark}
\newtheorem{remark}[theorem]{Remark}

\numberwithin{equation}{section}

\newcommand\nutwid{\overset {\text{\lower 3pt\hbox{$\sim$}}}\nu}

\allowdisplaybreaks

\newcommand\omycite[1]{}

\newcommand{\beqs}{\begin{equation*}}
\newcommand{\eeqs}{\end{equation*}}
\newcommand{\beq}{\begin{equation}}
\newcommand{\eeq}{\end{equation}}

\begin{document}
\title[A Generalization of the Amdeberhan-Andrews-Ballantine Conjecture]{A Generalization of the Amdeberhan-Andrews-Ballantine Conjecture}

\author{Rong Chen}
\address{Department of Mathematics, Shanghai Normal University, People's Republic of China}
\email{rchen@shnu.edu.cn}
\author{Tianjian Xu}
\address{Department of Mathematics, Shanghai University, People's Republic of China}
\email{xtjmath@shu.edu.cn}



\begin{abstract}
In this paper, we prove a generalization of a conjecture of Amdeberhan, Andrews, and Ballantine on double Lambert series. Motivated by a question raised by Cui, Kumar, and Singh concerning the existence of a generalization of this conjecture, we establish an identity in which the coefficients are given by the generalized divisor function $\sigma_k(n)$. As a special case, our result includes the original conjecture.
\end{abstract}
\maketitle
\section{Introduction}
In 1771, J. H. Lambert \cite{lam65} showed that the generating function of the divisor function can be expressed as
\begin{align*}
\sum_{n=1}^\infty\frac{q^n}{1-q^n}.
\end{align*}
This fundamental identity initiated the study of a broader class of series now known as Lambert series. In its classical form, a Lambert series is given by
\begin{align*}
\sum_{n=1}^\infty\frac{a_nq^n}{1-q^n}.
\end{align*}
In particular, when $a_n=n^{2k-1}$, the resulting Lambert series is a modular form for $k\geq2$ and is a quasimodular form when $k=1$. An excellent survey of existing key results and properties of Lambert series identities can be found in the work of Schmidt \cite{sch20}. Building on this classical framework, recent work has significantly broadened the scope of Lambert series theory. In particular, Amdeberhan, Andrews, and Ballantine \cite{amanba26} introduced an extensive generalization by considering generalized Lambert series of the form
\begin{align*}
\sum_{n\geq1}R_n(q^n,q),
\end{align*}
where $R_n(x,y)$ is a rational function of $x$ and $y$. Moreover, they developed a two-parameter extension called double Lambert series \cite{amanba26}, expressed by
\begin{align*}
\sum_{m,n\geq1}S_{n,m}(q^n,q^m,q),
\end{align*}
where $S_{n,m}(x,y,z)$ is a rational function of $x$, $y$, and $z$. They also established several connections between these series and Rogers-Ramanujan type $q$-series. In the course of their study, they formulated the following conjecture \cite{amanba26}.

\begin{conj}\cite[Conjecture 5.12]{amanba26}\label{AAB-conj-I}
Let $a$ be a positive integer. Then, for each positive integer $N$, we have
\begin{align*}
[q^{N2^a}]\sum_{m,n\geq 1}\frac{q^{mn2^a}}{(1+q^{n2^{a-1}})(1-q^{2m-1})}=\sigma_1(N),
\end{align*}
where $\sigma_1(n)$ is the sum of divisors of $n$.
\end{conj}

Recently, Cui, Kumar and Singh \cite{ks26} proved Conjecture \ref{AAB-conj-I}. In the conclusion of \cite{ks26}, Cui, Kumar and Singh asked whether there is a generalization of Conjecture \ref{AAB-conj-I} for the generalized divisor function
\begin{align*}
\sigma_k(n)=\sum_{d|n}d^k.
\end{align*}

In this paper, we first perform the following transformation on the left side of Conjecture \ref{AAB-conj-I}:
\begin{align*}
&\sum_{m,n\geq 1}\frac{q^{mn2^a}}{(1+q^{n2^{a-1}})(1-q^{2m-1})}\\
=&\sum_{m,n\geq 1}\frac{q^{mn2^a}}{(1-q^{n2^{a}})(1-q^{2m-1})}-\sum_{m,n\geq 1}\frac{q^{mn2^a+n2^{a-1}}}{(1-q^{n2^{a}})(1-q^{2m-1})}\\
=&\sum_{m,n\geq 1}\sum_{j\geq0}\frac{q^{mn2^a+nj2^a}}{1-q^{2m-1}}-\sum_{m,n\geq 1}\sum_{j\geq0}\frac{q^{mn2^a+n2^{a-1}+nj2^a}}{1-q^{2m-1}}\\
=&\sum_{m\geq 1}\sum_{j\geq0}\frac{q^{(m+j)2^a}}{(1-q^{2m-1})(1-q^{(m+j)2^{a}})}-\sum_{m\geq 1}\sum_{j\geq0}\frac{q^{(2m+2j+1)2^{a-1}}}{(1-q^{2m-1})(1-q^{(2m+2j+1)2^{a-1}})}\\
=&\sum_{n\geq 0}\sum_{m\geq n}\frac{q^{(m+1)2^a}}{(1-q^{2n+1})(1-q^{(m+1)2^{a}})}-\sum_{n\geq 0}\sum_{m\geq n}\frac{q^{(2m+3)2^{a-1}}}{(1-q^{2n+1})(1-q^{(2m+3)2^{a-1}})}.
\end{align*}
On this basis, we answer the question of Cui, Kumar and Singh by proving the following generalization of Conjecture \ref{AAB-conj-I} for the generalized divisor function $\sigma_k(n)$:
\begin{theorem}\label{g-AAB-conj-I}
Let $k$ and $a$ be positive integers. Then for each positive integer $N$, we have
\begin{align*}
[q^{N2^{a}}]\sum_{n=0}^\infty\sum_{m\geq n}\left(\frac{(m+1)^{k-1}q^{(m+1)2^a}}{(1-q^{2n+1})(1-q^{(m+1)2^a})}-\frac{(m-n+1)^{k-1}q^{(2m+3)2^{a-1}}}{(1-q^{2n+1})(1-q^{(2m+3)2^{a-1}})}\right)=\sigma_{k}(N).
\end{align*}
\end{theorem}
We can observe that setting $k=1$ in Theorem \ref{g-AAB-conj-I} is equivalent to Conjecture \ref{AAB-conj-I}. The proof of Theorem \ref{g-AAB-conj-I} follows a strategy similar to that developed in the authors' recent work on the double Lambert series \cite{cx26}.

\section{A generalization of Conjecture \ref{AAB-conj-I}}\label{main}
In this section, we prove Theorem \ref{g-AAB-conj-I}, which gives a generalization of Conjecture \ref{AAB-conj-I}. We first give the following identity, which is similar to \cite[Corollary 3.7]{cx26}.
\begin{lemma}\label{A-A=L}
For $|zq|<1$ and $x\neq q^{-r}$ for the positive integer $r$, we obtain
\begin{align*}
\sum_{n=0}^\infty\sum_{m\geq n}\frac{z^mq^m}{(1-q^{m+1})(1-xq^{n+1})}-\sum_{n=0}^\infty\sum_{m\geq n}\frac{xz^{m-n}q^{m+1}}{(1-xq^{m+2})(1-xq^{n+1})}=\sum_{n=0}^\infty\frac{(n+1)z^nq^n}{1-q^{n+1}}.
\end{align*}
\end{lemma}

\begin{proof}
Fix $N\geq 0$ and compare the coefficients of $z^N$ on both sides. For the right-hand side, we have
\begin{align*}
[z^N]\sum_{n=0}^\infty\frac{(n+1)z^nq^n}{1-q^{n+1}}=\frac{(N+1)q^N}{1-q^{N+1}}.
\end{align*}
For the left-hand side, we have
\begin{align*}
[z^N]\sum_{n=0}^\infty\sum_{m\geq n}\frac{z^mq^m}{(1-q^{m+1})(1-xq^{n+1})}=\frac{q^N}{1-q^{N+1}}\sum_{n=0}^N\frac{1}{1-xq^{n+1}}
\end{align*}
and
\begin{align}\label{2-sum}
[z^N]\sum_{n=0}^\infty\sum_{m\geq n}\frac{xz^{m-n}q^{m+1}}{(1-xq^{m+2})(1-xq^{n+1})}=\sum_{n\geq0}\frac{xq^{n+N+1}}{(1-xq^{n+N+2})(1-xq^{n+1})}.
\end{align}
We denote the infinite sum on the right-hand side of \eqref{2-sum} by $F(q)$. Then we have
\begin{align*}
F(q)=&\frac{q^N}{1-q^{N+1}}\sum_{n\geq0}\left(\frac{1}{1-xq^{n+1}}-\frac{1}{1-xq^{n+N+2}}\right)\\
=&\frac{q^N}{1-q^{N+1}}\lim_{M\rightarrow\infty}\sum_{n=0}^M\left(\frac{1}{1-xq^{n+1}}-\frac{1}{1-xq^{n+N+2}}\right)\\
=&\frac{q^N}{1-q^{N+1}}\lim_{M\rightarrow\infty}\left(\sum_{n=0}^N\frac{1}{1-xq^{n+1}}-\sum_{n=M+1}^{N+M+1}\frac{1}{1-xq^{n+1}}\right)\\
=&\frac{q^N}{1-q^{N+1}}\sum_{n=0}^N\frac{1}{1-xq^{n+1}}-\frac{q^N}{1-q^{N+1}}\lim_{M\rightarrow\infty}\sum_{n=0}^{N}\frac{1}{1-xq^{M+n+2}}\\
=&\frac{q^N}{1-q^{N+1}}\sum_{n=0}^N\frac{1}{1-xq^{n+1}}-\frac{(N+1)q^N}{1-q^{N+1}}.
\end{align*}
Therefore, we obtain
\begin{align*}
[z^N]\text{LHS}=\frac{(N+1)q^N}{1-q^{N+1}}=[z^N]\text{RHS}.
\end{align*}
\end{proof}

\begin{proof}[Proof of Theorem \ref{g-AAB-conj-I}]
We first consider the first series on the left-hand side:
\begin{align*}
&\sum_{n=0}^\infty\sum_{m\geq n}\frac{(m+1)^{k-1}q^{2^a(m+1)}}{(1-q^{2n+1})(1-q^{2^a(m+1)})}=\sum_{n=0}^\infty\sum_{m\geq n}\sum_{i=0}^\infty\frac{(m+1)^{k-1}q^{2^a(m+1)+(2n+1)i}}{1-q^{2^a(m+1)}}\\
=&\sum_{r=0}^{2^{a-1}-1}\sum_{n=0}^\infty\sum_{m\geq n}\sum_{i=0}^\infty\frac{(m+1)^{k-1}q^{2^a(m+1)+(2n+1)(2^{a-1}i+r)}}{1-q^{2^a(m+1)}}.
\end{align*}
We can observe that the power of $q$ in each term of the above series is congruent to $(2n+1)r$ modulo $2^{a-1}$. Therefore, we obtain
\begin{align*}
&[q^{N2^{a}}]\sum_{n=0}^\infty\sum_{m\geq n}\frac{(m+1)^{k-1}q^{2^a(m+1)}}{(1-q^{2n+1})(1-q^{2^a(m+1)})}=[q^{N2^{a}}]\sum_{n=0}^\infty\sum_{m\geq n}\sum_{i=0}^\infty\frac{(m+1)^{k-1}q^{2^a(m+1)+2^{a-1}(2n+1)i}}{1-q^{2^a(m+1)}}\\
&=[q^{2N}]\sum_{n=0}^\infty\sum_{m\geq n}\sum_{i=0}^\infty\frac{(m+1)^{k-1}q^{2(m+1)+(2n+1)i}}{1-q^{2m+2}}=[q^{2N}]\sum_{n=0}^\infty\sum_{m\geq n}\frac{(m+1)^{k-1}q^{2m+2}}{(1-q^{2m+2})(1-q^{2n+1})}.
\end{align*}
Similarly, we obtain
\begin{align*}
[q^{N2^{a}}]\sum_{n=0}^\infty\sum_{m\geq n}\frac{(m-n+1)^{k-1}q^{2^{a-1}(2m+3)}}{(1-q^{2n+1})(1-q^{2^{a-1}(2m+3)})}=[q^{2N}]\sum_{n=0}^\infty\sum_{m\geq n}\frac{(m-n+1)^{k-1}q^{2m+3}}{(1-q^{2n+1})(1-q^{2m+3})}.
\end{align*}
Thus the desired identity is equivalent to
\begin{align}\label{q2n-target}
\sum_{n=0}^\infty\sum_{m\geq n}\frac{(m+1)^{k-1}q^{2m+2}}{(1-q^{2m+2})(1-q^{2n+1})}-\sum_{n=0}^\infty\sum_{m\geq n}\frac{(m-n+1)^{k-1}q^{2m+3}}{(1-q^{2n+1})(1-q^{2m+3})}=\sum_{n=1}^\infty\sigma_k(n)q^{2n}.
\end{align}

Multiplying both side of Lemma \ref{A-A=L} by $zq$, replacing $q$ with $q^2$, and setting $x=q^{-1}$, we obtain
\begin{align}\label{in-lemma1}
\sum_{n=0}^\infty\sum_{m\geq n}\frac{z^{m+1}q^{2m+2}}{(1-q^{2m+2})(1-q^{2n+1})}-\sum_{n=0}^\infty\sum_{m\geq n}\frac{z^{m-n+1}q^{2m+3}}{(1-q^{2n+1})(1-q^{2m+3})}=\sum_{n=0}^\infty\frac{(n+1)z^{n+1}q^{2n+2}}{1-q^{2n+2}}.
\end{align}
For the right-hand side, we obtain
\begin{align*}
\sum_{n=0}^\infty\frac{(n+1)z^{n+1}q^{2n+2}}{1-q^{2n+2}}=\sum_{n=1}^\infty\frac{nz^nq^{2n}}{1-q^{2n}}=\sum_{n,k\geq1}nz^nq^{2nk}=\sum_{n=1}^\infty\Big(\sum_{d|n}dz^d\Big)q^{2n}.
\end{align*}
Then applying $z\frac{d}{dz}$ to both sides of \eqref{in-lemma1} $k-1$ times, we obtain
\begin{align*}
\sum_{n=0}^\infty\sum_{m\geq n}\frac{(m+1)^{k-1}z^{m+1}q^{2m+2}}{(1-q^{2m+2})(1-q^{2n+1})}-\sum_{n=0}^\infty\sum_{m\geq n}\frac{(m-n+1)^{k-1}z^{m-n+1}q^{2m+3}}{(1-q^{2n+1})(1-q^{2m+3})}=\sum_{n=1}^\infty\Big(\sum_{d|n}d^kz^d\Big)q^{2n}.
\end{align*}
Setting $z=1$ in the above identity, we complete the proof of \eqref{q2n-target}.
\end{proof}
\begin{remark}
We finally point out that the generalization obtained in this paper is expressed as the coefficient of the difference of two series. It would be interesting to find, if possible, a further generalization of Conjecture \ref{AAB-conj-I} in which $\sigma_k(n)$ arises as the coefficient of a single series.
\end{remark}
\subsection*{Acknowledgements}
The first author was supported by the National Natural Science Foundation of China (Grant No. 12401438).


\end{document}